\documentclass[12pt,one side,]{amsart}
\usepackage{amssymb,latexsym,xy,eucal,mathrsfs, tikz}
\baselineskip 24 truept \numberwithin{equation}{section}
\newtheorem{theorem}{Theorem}[section]
\newtheorem{lemma}[theorem]{Lemma}

\newtheorem{proposition}[theorem]{Proposition}
\newtheorem{corollary}[theorem]{Corollary}
\newtheorem{definition}[theorem]{Definition}
\newtheorem{example}[theorem]{Example}

\usepackage{setspace}
\usepackage[margin=1in]{geometry}
\usepackage{devanagari}
\begin{document}
\baselineskip 15 truept
\title{ Strong zero-divisor graph of p.q.-Baer $*$-rings }	\date{}

\author{Anil Khairnar}
	
\address{ \textit{Department of Mathematics, Abasaheb Garware College, Pune-411004, India.\\ anil\_maths2004@yahoo.com}}
	
\author{Nana Kumbhar}
	
\address{\textit{Department of Mathematics and Statistics, Yashwantrao Mohite College, Bharati Vidyapeeth (Deemed to be University), Pune-411038, India.\\ kumbharnana03@gmail.com}}
	
\author{ B. N. Waphare }
\address{\textit{Department of Mathematics, Savitribai Phule Pune University, Pune-411007, India. \\ waphare@yahoo.com}}
	
\maketitle {\bf \small Abstract:}{ In this paper, we study the strong zero-divisor graph of a  p.q.-Baer $*$-ring. We determine the condition on a p.q.-Baer $*$-ring (in terms of the smallest central projection in a lattice of central projections of a $*$-ring), so that its strong zero-divisor graph contains a cut vertex. It is proved that the set of cut vertices of a strong zero-divisor graph of a p.q.-Baer $*$-ring forms a complete subgraph. We prove that the complement of the strong zero-divisor graph of a p.q.-Baer $*$-ring is connected if and only if the $*$-ring contains at least six central projections. We characterize the diameter and girth of the complement of a strong zero-divisor graph of a p.q.-Baer $*$-ring. Also, we characterize p.q.-Baer $*$-rings whose strong zero-divisor graph is complemented.}
	
	\section{Introduction}

A $\ast-ring$ $R$ is a ring with an involution $x \rightarrow x^*$, that is, an additive anti-automorphism of the period at most two. An element $e$ of
	a $*$-ring is a \textit{projection} if it is idempotent  (\textit{i.e.} $e^2=e$)  and self adjoint (\textit{i.e.} $e^*=e$). We write $r_R(B)=\{x\in R\colon bx=0,\forall~ b\in B\}$ and is called the \textit{right annihilator} of $B$ in $R$, and  $l_R(B)=\{x\in R\colon xb=0,\forall~ b\in B\}$ and is called the \textit{left annihilator} of $B$
	in $R$. A $*-ring$ $R$ is said to be a \textit{Baer $*$-ring}, if for any $B\subset R$, $r_R(B)=eR$, where $e$ is a projection in $R$. A $*-ring$ $R$ is said to be a \textit{Rickart $*$-ring}, if for any $x \in R$, $r_R(x)=eR$, where $e$ is a projection in $R$. 
	For each element $a$ in a Rickart $*$-ring, there is a unique projection $e$ such that $ae=a$ and $ax=0$ if and only if $ex=0$, called a {\it right projection} of $a$ denoted by $RP(a)$. Similarly, left
	projection $LP(a)$ is defined for each element $a$ in a Rickart $*$-ring $R$. In fact, $r_R(a)=(1-RP(a))R$. A subring $B$ of a $*$-ring $R$ is a \textit{$*$-subring} if $x\in B$ implies $x^*\in B$. A projection in a $*$-ring is a \textit{central projection} if it belongs to the centre of the ring. More details about Baer $*$-ring and  Rickart $*$-rings can be found in  Berberian \cite{skb}.\

A $*$-ring $R$ is said to be \textit{quasi Baer $*$-ring} \cite{Bir2}, if the right annihilator of every ideal of $R$ is generated by a projection in $R$. A $*$-ring $R$ is said to be a \textit{p.q.-Baer $*$-ring} \cite{Bir4}, if for every principal right ideal $aR$ of $R$, $r_R(aR)=eR$, where $e$ is a projection in $R$.  In \cite{nkbn2}, Waphare and Thakare proved that, $M_n(Z_m)$ is Baer $*$-ring for $n \geq 2$ if and only if $n=2$ and $m$ is square free integer whose every prime factor is of the form $4k+3 $. So, in general $M_n(R)$ need not be Baer $*$-ring for Baer $*$-ring $R$. It observed that, $R$ is quasi-Baer $*$-ring if and only if $M_n(R)$ ($n\times n $ matrix ring with $*$-transpose as involution) is a quasi-Baer $*$-ring (hence a p.q.-Baer $*$-ring) for all $n\geq 1$ \cite[Proposition 2.6]{Bir2}. Hence the class of p.q.-Baer $*$-rings is much larger than the class of Baer $*$-rings. Let $R$ be a $*$-ring, $a\in R$, we say that $a$ possesses a \textit{central cover} if there exists a smallest central projection $h$ such that $ha=a$, we denote central cover of $a$ by $C(a)$.\

Graph with no loops and no multiple edges is called a \textit{simple graph}. Graph $G$ is said to be \textit{complete} if there is an edge between any two distinct vertices. If any two distinct vertices are joined by a path then $G$ is called a \textit{connected graph}. If $x$ and $y$ are two vertices of $G$, $d(x, y)$ denoted the length of the shortest path joining $x$ and $y$. The \textit{diameter} of $G$ is defined as $diam(G)= sup \{d(x,y)~|~x ~and~ y ~are~ vertices~ of ~G \}$. The \textit{girth} of $G$ is the length of the shortest cycle of $G$ and it is denoted by \textit{gr(G)}. Vertex $a$ of a graph $G$ is called a \textit{cut vertex} if the removal of $a$ and all edges incident on a creates a graph with more connected components than $G$.
	We say that $G$ is \textit{split} into two subgraphs $X$ and $Y$ via $a$ if (i) $|V(X)|\geq 2$ and $|V(Y)|\geq 2$, (ii) $X \cup Y = G$, (iii) $V(X)\cap V(Y) =\{a\} $ (iv) $a$ is cut vertex of $G$, and (v) $x\in V(X)\backslash \{a\}$ and $y \in V(Y)\backslash \{a\}$, then $x$ and $y$ are not adjacent. For more about graph theoretic concepts, see \cite{dw}.

Relating one structure to another and studying the interplay between them is an interesting area of research. Researchers are attracted to this area as one can solve problems in one structure using another. In \cite{Abrams}, G. Abrams and G. A. Pion, were introduced the Leavitt path algebra of a graph as the algebraic counterpart of the graph $C^*$- algebra and as generalizations of Leavitt algebras. The ideas or process in commutative rings becomes quite difficult for non-commutative rings, in particular $*$-rings. Giving examples in $*$-rings with certain properties is a difficult task. In \cite{Lia}, R. Hazrat and L. Va\v{s} characterize Leavitt path algebras, which are Rickart, Baer and Baer $*$-rings in terms of the properties of the underlying graph.

Researchers relates a graph to  an algebraic structure. The zero-divisor graph of a commutative ring was introduced by Beck in \cite{B}. The {\it chromatic number} $\chi(G)$ of a graph $G$ is the minimum number of colors that can be assigned to the vertices of $G$ in such a way that any two adjacent vertices have different colors. A subset $S$ of the vertex set $V(G)$ of $G$ is called a {\it clique} if any two distinct vertices of $S$ are adjacent; the number $\omega(G)$ is the least upper bound of the size of the cliques in $G$. Beck conjectured that for commutative ring $R$,   $\chi(\Gamma(R))=\omega(\Gamma(R))$, if $\omega(\Gamma(R))< \infty$. In \cite{ddn},  D.D. Anderson and Naseer, gave a counterexample of a commutative ring for which the conjecture is not true. D.F. Anderson and Livingston, redefined the zero-divisor graph in \cite{DP}, and proved such a graph is connected and diameter less equal $3$. Properties of $\Gamma(R)$ which are in terms of connectedness, diameter, girth, etc., are investigated by different authors see \cite{saam1,saam,DP}.  In \cite{spr}, Redmond generalized the concept of the zero-divisor graph to non-commutative rings. Associate directed graph $\Gamma(G)$ for non-commutative ring $R$, its vertices are all non zero zero-divisors of $R$ in which for any two distinct vertices $x$ and $y$, $x \rightarrow y$ is an edge if and only if $xy = 0$. In \cite{saam}, Akbari and Mohammadin associate simple undirected graph $\bar{\Gamma}(R)$ for non-commutative ring $R$, in which two vertices $x$ and $y$ are adjacent if and only if $x \neq y$ and either $xy = 0$ or $yx = 0$. In \cite{Behboodi1}, M. Behboodi \textit{et al.}, introduced \textit{strong zero-divisor}, an element $a\in R$ is called strong zero-divisor in a ring $R$ if  $\langle a \rangle \langle b\rangle =0$ or $\langle b \rangle \langle a \rangle =0$  for some nonzero element $b\in R$. The set of all \textit{strong zero-divisors} in a ring $R$ is denoted by $S(R)$. In \cite{Behboodi2}, M. Behboodi \textit{et al.} associated a graph $\tilde{\Gamma}(R)$ to a ring $R$  whose vertices are  $S(R)^*=S(R)\backslash\{0\}$, where distinct vertices $a$ and $b$ are adjacent if and only if either $\langle a \rangle \langle b\rangle =0$ or $\langle b \rangle \langle a \rangle =0$. This graph is called the \textit{strong zero-divisor graph} of a ring $R$.  In the same paper, they studied the interplay between the ring theoretic properties of $R$ and the graph theoretic properties of $\tilde{\Gamma}(R)$. The zero-divisor graph of a ring is generalized to the \textit{$k$-zero-divisor hypergraph} of a ring $R$, for $k \in \mathbb{N}$, which is denoted by $H_k(R)$ by T. Asir \textit{et al.} see, \cite{Asir}. They determined the diameter and the girth of $H_k(R)$, whenever $R$ is reduced ring.

	In \cite{Pat1}, Patil and Waphare associated a simple undirected graph to a $*$-ring $R$ whose vertices are the nonzero left zero-divisors, {\it i.e.}, $\{0\neq a\in R~|~ ab=0,\textnormal{ for some nonzero } b\in R\}$ and two distinct vertices $a$ and $b$ are adjacent if and only if $ab^*=0$. This graph is called the {\it zero-divisor graph} of the $*$-ring $R$ and denote it by $\Gamma^*(R)$. They investigate properties of $\Gamma^*(R)$ which are in terms of connectedness, diameter and girth. Also they proved Beck's conjecture is true for $\Gamma^*(R)$.  In \cite{Pat2}, Patil and Waphare proved that the set of cut vertices forms a complete subgraph. Also, characterize Rickart $*$-rings for which the complement of the $\Gamma^*(R)$ is connected. Axtell \textit{et al.} \cite{Axtell1} determine when the zero-divisors form an ideal in a finite commutative ring with identity. Axtell \textit{et al.} \cite{Axtell2}  and S.P. Redmond  \cite{spr2}, examined and characterized the existence of cut vertex in the zero-divisor graph of commutative ring. In \cite{spr2}, S.P. Redmond proved that the set of cut vertices of the zero-divisor graph of a commutative ring forms a complete subgraph.  Patil and Waphare \cite{Pat2}, proved analogous results for zero-divisor graph of Rickart $*$-rings.  Vishweswaran \cite{Visweswaran1}, studied the complement and connectedness of complement of the zero-divisor graph of commutative rings.
	
	In \cite{NMK1}, authors introduced the \textit{strong zero-divisor graph} for a $*$-ring $R$. Associate  a simple undirected graph to a $*$-ring $R$ whose vertex set is   $V(\Gamma_s^*(R))=\{0\neq a\in R$ $|$ $r_R(aR)\neq \{0\} \}$ and two distinct vertices $a$ and $b$ are adjacent  if and only if $aRb^*=0 $. This graph is called the  \textit{strong zero-divisor graph} of a $*$-ring $R$ and denoted by $\Gamma_s^*(R)$. In general $\Gamma_s^*(R)$ is not connected. In the same paper, authors studied the girth, diameter, and connectedness of $\Gamma_s^*(R)$.  It is proved that for a p.q.-Baer $*$-ring $R$, $\Gamma_s^*(R)$ is uniquely complemented.
	
	In this paper, we continue the study of a strong zero-divisor graph of rings with involution. 
	Axtell \textit{et al.} \cite{Axtell1} examined the cut vertex in the sense of zero-divisor graphs. In the second section, we study the existence of a cut vertex in the $\Gamma_s^*(R)$, when $R$ is a $*$-ring. This extends the results of Axtell \textit{et al.} \cite{Axtell2} to $*$-ring. In section three, we prove that the set of cut vertices of $\Gamma_s^*(R)$, forms a complete subgraph when $R$ is a p.q.-Baer $*$-ring. In the fourth section, we study the complement of a strong zero-divisor graph of p.q.-Baer $*$ ring.  We prove that the complement of the strong zero-divisor graph of a p.q.-Baer $*$-ring is connected if and only if the ring contains at least six central projections. Also, we characterize the diameter and girth of the complement of strong zero-divisor graph of a p.q.-Baer $*$-rings.

	\section{Strong zero-divisor graph of $*$-rings}
	The strong zero-divisors for a $*$-ring were introduced by authors in \cite{NMK1}. In this section, we examine cut vertices of $\Gamma_s^*(R)$. Henceforth we assume that $V(\Gamma^*_s(R))\neq \emptyset$, call $\Gamma^*_s(R)$ the  \textit{strong zero-divisor graph of a $*$-ring} $R$ and  $x \sim y$ means $x, y$ are adjacent.
	
	\begin{definition} [{\cite[Definition 2.1]{NMK1}}] \label{def1} Let $R$ be a $*$-ring. We associate a simple undirected graph $\Gamma_s^*(R)$ to $R$ whose vertex set is  $V(\Gamma_s^*(R))=\{0\neq a \in R$ $|$ $r_R(aR)\neq \{0\} \}$ and two distinct vertices $a$ and $b$ are adjacent if and only if $aRb^*=0$. 
	\end{definition}

	\begin{example}
		Let $R=M_2(\mathbb{Z}_{6})$, and let 
		$a=\left[
		\begin{array}{ll}
			0 & 2 \\
			0 & 2\\
		\end{array}
		\right]$,
		$b=\left[\begin{array}{ll}
			2 & 0 \\
			2 & 0\\
		\end{array}
		\right]$,
		$c=\left[
		\begin{array}{ll}
			3 & 3 \\
			3 & 3\\
		\end{array}
		\right]  \in R$. Let $*$ be the transpose involution on $R$. Clearly, $aRc^*=0$ and $bRc^*=0$ i.e. $a, b\in V(\Gamma^*_s(R))$. On the other hand $ab^*=0$, so $a, b\in V(\Gamma^*(R))$ but $aRb^*\neq 0$, and hence, the vertex $a$ and $b$ are adjacent in $ \Gamma^*(R)$ but not in $\Gamma^*_s(R)$.
	\end{example}

	\begin{example}
		Let $R=M_2(\mathbb{Z}_{6})$, and let 
		$a=\left[
		\begin{array}{ll}
			0 & 2 \\
			0 & 2\\
		\end{array}
		\right]$,
		$b=\left[
		\begin{array}{ll}
			2 & 2 \\
			0 & 0\\
		\end{array}
		\right]$,
		$c=\left[
		\begin{array}{ll}
			3 & 3 \\
			3 & 3\\
		\end{array}
		\right]  \in R$. Let $*$ be the transpose involution on $R$. Clearly, $aRc^*=0$ and $bRc^*=0$ i.e. $a, b\in V(\Gamma^*_s(R))$. On the other hand $ab=0$, so $a, b\in V(\bar{\Gamma}(R))$ but $aRb^*\neq 0$, and hence, the vertex $a$ and $b$ are adjacent in $\bar{\Gamma}(R)$ but not in $\Gamma^*_s(R)$.
	\end{example}
	From the above examples, it is clear that the general graph $\Gamma^*_s(R)$ is not an induced subgraph of $\bar{\Gamma}(R)$ and  $\Gamma^*(R)$.

	\begin{theorem} [{\cite[Theorem 2.1]{NMK1}}]{\label{th1}} Let $R_1$, $R_2$ be two $*$-rings with  $V(\Gamma^*_s(R_1))\ne \emptyset$ and $V(\Gamma^*_s(R_2))\neq \emptyset$. Then $A=R_1\oplus R_2$ is a $*$-ring with a componentwise involution such that $\Gamma^*_s(A)$ is connected and $diam(\Gamma^*_s(A))\leq 4$.
	\end{theorem}

	\begin{corollary}{\label{cor1}}
		If $A_1, A_2, A_3, \cdots A_n$ are $*$-subrings of a $*$ ring $A$ such that $A=A_1\oplus A_2 \oplus \cdots \oplus A_n$, then $\Gamma_s^*(A)$ is connected.
	\end{corollary}
	\begin{corollary}{\label{cor2}}
		Let $R$ be a $*$-ring.  If $R$ contains a nontrivial central projection, then $\Gamma _s^*(R)$ is connected.
	\end{corollary}

	Axtell \textit{et al.} \cite{Axtell2} and Redmond \cite{spr2} characterized the existence of a cut vertex in the zero-divisor graph of a commutative ring.  Axtell \textit{et al.} \cite[Theorem 4.4]{Axtell1}, proved that $\{0,a\}$ is ideal in finite commutative ring $R$ if  $a$ is cut vertex in $V(\Gamma(R))$. We prove its analogous for $\Gamma^*_s(R)$ in the following theorem.
	
	\begin{theorem}{\label{thm2}}
		Let $R$ be a $*$-ring and $\Gamma^*_s(R)$ splits into two subgraphs $X$ and $Y$ via $a$. Then $\{0, a\}$ is an ideal of $R$.
	\end{theorem}
	
	\begin{proof}
		Clearly $a\neq a+a$ and $aRx^*=0$ implies $(a+a)Rx^*=0$. We claim that $a+a=0$. If $a+a\neq 0$, then $a+a\in V(\Gamma^*_s(R))$. Assume that  $a+a\in V(X)\backslash\{a\}$. Let  $b\in V(Y)\backslash \{a\}$ be such that $aRb^*=0$, then $(a+a)Rb^*=0$ a contradiction to fact that $a$ is a cut vertex, therefore $a+a=0\in \{0,a\}$. If $r\in R$, then $aRc^*=0$ implies $raRc^*=0$ and $arRc^*=0$ . If $ra\neq 0$ and $ra\neq a$ then, each vertex adjacent to $a$ is also adjacent to $ra$, a contradiction to $a$ is a cut vertex. Hence $ra\in \{0,a\}$, similarly $ar\in \{0,a\}$. Therefore $\{0,a\}$ is an ideal of $R$.
	\end{proof}
	\begin{theorem}{\label{thm3}}
		Let $R$ be a $*$-ring and $\Gamma^*_s(R)$ splits into two subgraphs $X$ and $Y$ via $a$ such that $X\backslash \{a\}$ is complete subgraph, then $V(X)\cup\{0\}$ is an ideal of $R$.
	\end{theorem}
	\begin{proof}
		Let $b\in X\backslash \{a\}$ such that $a$ and $b$ are adjacent. As $X\backslash \{a\}$ is complete, for any $x, y\in V(X) \cup \{0\}$, $bRx^*=0$ and $bRy^*=0$. So, $bR(x+y)^*=bR(x^*+y^*)=bRx^* +bRy^*=0$ and  hence $x + y\in V(X)\cup \{0\}$. If $r\in R$ then $bR(rx)^*=bR(x^*r^*)=0$, which gives $rx\in V(X)\cup \{0\}$. Also, if $r\in R$ then $bR(xr)^*=bR(r^*x^*)=0$, which gives $xr\in V(X)\cup \{0\}$. Hence $V(X)\cup \{0\}$ is an ideal of $R$.
	\end{proof}
	Converse of Theorem \ref{thm3}, is not true. Let $R= \mathbb{Z}_2 \times \mathbb{Z}_4$ be the $*$-ring with identity mapping as involution. Here $X=\{(1,0),(0,2),(1,2)\}$, $Y= (0,1),(1,0),(0,3)$ and $(1,0)$ is cut vertex and $ V(X)\cup \{(0,0)\}=\{(0,0),(1,0),(0, 2),(1,2)\}$ is ideal of $R$  but $X$ is not complete subgraph.

	\begin{center}
		\begin{tikzpicture}[scale=.6]     
			
			\draw (0,4)--(2,2)--(4,2)--(6,2);
			\draw (0,0)--(2,2);
			\draw[fill=black](0,4) circle(.1);
			\draw[fill=black](0,0) circle(.1);
			\draw[fill=black](2,2) circle(.1);
			\draw[fill=black](4,2) circle(.1);
			\draw[fill=black](6,2) circle(.1);
			\node [left ] at(0,4) {$(0,1)$};
			\node [left ] at(0,0) {$(0,3)$};
			\node [below] at(2.4,2) {$(1,0)$};
			\node [below] at(4.4,2) {$(0, 2)$};
			\node [below] at(6.4,2) {$(1,2)$};
			\node [below] at (3, -1){$Fig. 1.$};
			
		\end{tikzpicture}  
	\end{center}

	\begin{proposition}
		Let $R$ be a $*$-ring. If $\Gamma^*_s(R)$ contains a cut vertex, then $x+ y\in V(\Gamma^*_s(R))\cup \{0\}$ for all $x, y \notin V(\Gamma^*_s(R))\cup \{0\}$.
	\end{proposition}

	\begin{proof} 
		Let $a$ be a cut vertex of $\Gamma^*_s(R)$. By Theorem \ref{thm2}, $\{0, a\}$ is a left ideal in $R$, this gives $Ra=I=\{0, a\}$, and for any $x\in R$, $xRa=xI =\{0, xa\}  \subseteq I$. If $x\in R\backslash V(\Gamma^*_s(R))$ then $xRa\neq 0$, hence $x Ra\subseteq I$. Similarly $y\notin V(\Gamma^*_s(R))$ implies  $-y\notin V(\Gamma^*_s(R))$ implies $-yRa \subseteq I$. This gives $xRa=-yRa=I$, implies $(x+y)Ra=0$. Therefore, $x+y\in V(\Gamma^*_s(R))\cup \{0\}$.
	\end{proof}
	The converse of the above statement is not true. Let $R=\mathbb{Z}_4$ with identity mapping as an involution. Here $V(\Gamma^*_s(R))=\{2\}$ and for any $x, y \notin V(\Gamma^*_s(R))\cup \{0\}$, $x+y\in V(\Gamma^*_s(R))\cup \{0\}$ but $\Gamma^*_s(R)$ does not have a cut vertex.

	\begin{definition}
		If $a\neq 0$, we  say that $r_R(aR)$ (right annihilator of $aR$) is properly maximal if whenever $b\in R\backslash \{0, a\}$, we have $r_R(aR)\nsubseteq r_R(bR)$.
	\end{definition}
	
	\begin{proposition}{\label{pro1}}
		Let $R$ be a $*$-ring and $a$ is a cut vertex in $\Gamma^*_s(R)$, then  $r_R(aR)$ is properly maximal.
	\end{proposition}
	
	\begin{proof}
		Suppose there exists $b\in R\backslash \{0,a\}$ such that $r_R(aR)\subseteq r_R(bR)$. Then every vertex adjacent to $a$ is also adjacent to $b$. Hence $a$ is not a cut vertex, a contradiction. Hence $r_R(aR)$ is properly Maximal.
	\end{proof}
	The following example shows that the converse of the above proposition is not true.
	\begin{example}
		Let $R=\mathbb{Z}_6$ with identity mapping as an involution. $r_R(4)=\{0, 3\}$ is properly maximal, but $4$ is not a cut vertex. 
	\end{example}

	\section{Strong zero-divisor graph of p.q.-Baer $*$-rings}
	
	In \cite{Bir4}, Birkenmeier \textit{et al.} introduced principally quasi-Baer(p.q.-Baer) $*$-rings. $CP(R)$ is the set of all central projections in $*$-ring $R$. An involution $*$ of a ring $R$ is said to be \textit{semiproper} if for any $a\in R$, $aRa^* = \{0\}$ implies $a = 0$. The class of p.q.-Baer $*$-rings is larger than the class of Baer $*$-rings and Rickart $*$-rings, see {\cite[Exercise 10.2.24.4, Exercise 10.2.24.5]{Bir4}}, {\cite[Example 2.3, Example 2.6]{Anil4}}. In p.q.-Baer $*$-ring, $aRb^*=0$ if and only if $C(a)C(b)=0$. Mostly, we use this equivalent condition for adjacency in $\Gamma^*_s(R)$. Note that, for a p.q.-Baer $*$-ring  $R$, $V(\Gamma^*_s(R))=\emptyset$, whenever $CP(R)=\{0, 1\}$. Henceforth, we assume that a p.q.-Baer $*$-ring  $R$ contains at least four Central projections, i.e. $|CP(R)|\geq 4$.
	
	In \cite{Axtell1}, Axtell proved the following result.
	\begin{theorem}[{\cite[Theorem 4.3]{Axtell1}}]{\label{thmc4}}
		Let $R$ be a commutative ring with identity such that $\Gamma(R)$ is partitioned into two subgraphs $X$ and $Y$ with cut vertex $a$ and $|X|>2$. If $X$ is a complete subgraph, then every vertex of $X$ is looped.
	\end{theorem}

	The following theorem is analogous to the Theorem \ref{thmc4}.
	
	\begin{theorem}{\label{thm6}}
		Let $R$ be a p.q.-Baer $*$-ring such that  $\Gamma^*_s(R)$ splits into two subgraphs $X$ and $Y$ via a such that $|V(X)|>2$. If $X$ is a complete subgraph, then $b^2\in \{0,b\}$, for all $b\in V(X)\backslash \{a\}$.
	\end{theorem}
	\begin{proof}
		Let $b\in V(X)\backslash\{a\}$, with $b^2\notin \{0, b\}$. Since $X$ is complete subgraph, we have $cRb^*=0$, $\forall ~ c\in V(X)\backslash\{a, b\}$. Hence $cR(b^*)^2=cR(b^2)^*=0$ this gives $b^2\in V(X)$ and $cR(b^2+b)^* =0$. If $b^2+b=0$ then  $b^2=-b$, $b^2\neq b$ therefore $b^2$ is adjacent to $b$, implies $0=b^2Rb^*=-bRb^*=bRb^*$. Since involution in p.q.-Baer $*$-ring is semiproper, we have $b=0$, a contradiction. Hence $b^2+b\neq 0$. Thus $b^2+b\in V(X)$. If $b^2+b=b$ then $b^2=0\in \{0, b\}$, a contradiction. If $b^2+b\notin \{0, b\}$, then $0=(b^2+b)R(b)^*=bRb^*$, again a contradiction. Therefore $b^2\in\{0, b\}$, $\forall$  $b\in V(X)\backslash \{a\}$.	
	\end{proof}
	
	\begin{corollary}
		Let $R$ be a p.q.-Baer $*$-ring such that $\Gamma^*_s(R)$ splits into two subgraphs $X$ and $Y$ via $a$ such that $|V(X)|>2$. Then the following statements holds.
		\begin{itemize}
			\item[(a)] If there is $b \in V(X)\backslash \{a\}$ with $b^2\notin \{0, b\}$, then $X$ cannot be complete.
			\item[(b)] If $X$ is complete subgraph, then $V(X)\cup \{0\}$ is an ideal of $R$.
		\end{itemize}
	\end{corollary}
	\begin{proof}
		(a) Follow from Theorem {\ref{thm6}}.\\
		(b) Let $b\in V(X)$. Then for any $c\in V(X)\backslash \{b\}$, we have $bRc^*=0$, which gives $(b+b)Rc^*=0$. If $b+b\neq 0$, then $b+b\in V(X)\subseteq V(X)\cup \{0\}$. Let $x, y \in V(X)\cup \{0\}$, if $x=0$ or $y=0$, then $x+y\in V(X)\cup \{0\}$. Let $x=b$, $y\neq b$, then $(b+y)Rc^*=0$, for any $c\in V(X)\cup \{0\}$. Hence $x+y\in V(X)\cup \{0\}$. Let $x\neq b$, and  $y \neq b$. Then $xRb^*=0=yRb^*$ gives $(x+y)Rb^*=0$. Therefore  $x+y\in V(X)\cup \{0\}$. Let $r\in R$, then $xRb^*=0$ implies $rxRb^*=0$ and $xrRb^*=0$. Thus $V(X)\cup \{0\}$ is an ideal of $R$.
	\end{proof}
	Set of all central projections in p.q.-Baer $*$-ring forms a lattice under the partial order, $``e \leq f$ if and only if $e = ef = fe"$ and $e \wedge f = ef$ and $e\vee f = e + f-ef$, it is denoted by $L(CP(R))$. An element $e$ in a lattice is an atom if $0\leq f \leq e$ implies either $f=0$ or $f=e$.

	\begin{theorem}{\label{thm4}}
		Let $R$ be a p.q.-Baer $*$-ring and $a$ is cut vertex of  $\Gamma^*_s(R)$, Then $a$ is an atom in the lattice of central projections in $R$.
	\end{theorem}

	\begin{proof}
		First we show that $a$ is a central projection, Let $e=C(a)$, we have $r_R(aR)=r_R(eR)$. If $a\neq e$, then  $e\in R\backslash \{0, a\}$, such that $r_R(aR)\subseteq r_R(eR)$, a contradiction to fact that $a$ is cut vertex. Hence $a=C(a)=e$ \textit{i.e.} $a$ is a central projection. Next, let $f$ be a central projection in $R$ such that $f\leq a$ i.e. $f=af=fa$. Then $r_R(aR)\subseteq r_R(fR)$. Since $r_R(aR)$ is properly maximal, we get either $f=0$ or $f=a$. Therefore $a$ is an atom in the lattice of central projections in $R$. 	
	\end{proof}
	
	The converse of the above theorem is not true.
	\begin{example}
		Let $R=M_2(\mathbb{Z}_6)$ with identity mapping as an involution. Then $R$ is a p.q.-Baer $*$-ring such that 	
		$ CP(R)={\left\{
			0= \begin{bmatrix}  0  & 0 \\  0 &  0\end{bmatrix}, 
			1= \begin{bmatrix}  1  & 0 \\  0 &  1\end{bmatrix}, 
			e= \begin{bmatrix}  3  & 0 \\  0 &  3\end{bmatrix}, 
			f= \begin{bmatrix}  4  & 0 \\  0 &  4\end{bmatrix}
			\right \}}$ are the central projections in $R$. Here, $\Gamma^*_s(R) \cong  K_{80, 15}$  and
		  $e= \begin{bmatrix}  3  & 0 \\  0 &   3 \end{bmatrix}$ is an atom in $L(CP(R))$ but $e$ is not cut vertex in $\Gamma^*_s(R)$.
		\begin{center}
			\begin{tikzpicture}[scale=.6]    
				\draw (0,4)--(2,2)--(0,0)--(-2,2)--(0,4);
				\draw[fill=black](0,4) circle(.1);
				\draw[fill=black](0,0) circle(.1);
				\draw[fill=black](2,2) circle(.1);
				\draw[fill=black](-2,2) circle(.1);
				\node [above] at(0,4) {$1$};
				\node [below] at(0,0) {$0$};
				\node [left] at (-2,2) {$e$};
				\node [right] at (2,2) {$f$};
				\node [left] at (-2,2) {$e$};
				\node [below] at(0,-1) {$Fig. 2$};
			\end{tikzpicture}  
		\end{center}
	\end{example}
	\noindent Vertex of degree $1$ is called pendent vertex. Note that, if $a$ is pendent vertex in $\Gamma^*_s(R)$ then $a = c(a)$. 
	
	\noindent Here we provide a necessary and sufficient condition so that $\Gamma^*_s(R)$ contains a cut vertex.
	\begin{theorem}{\label{thm5}}
		Let $R$ be p.q.-Baer $*$-ring. Then $\Gamma^*_s(R)$ contains a cut vertex if and only if it contains a pendent vertex.
	\end{theorem}
	\begin{proof}
		Let $a$ be a cut vertex in $\Gamma^*_s(R)$. By Theorem \ref{thm4}, $a$ is an atom in $L(CP(R))$. Also, $a$ is adjacent to $1-a$. Let $b$ be a vertex adjacent to $1-a$, Hence $(1-a)RC(b)=0$, this gives $(1-a)C(b)=0$, i.e. $C(b)\leq a$ in $L(CP(R))$. Since $a$ is an atom and $C(b)\neq 0$, this implies $C(b)=a$. Since $r_R(bR)=r_R(C(b)R)$ i.e. $r_R(bR)=r_R(aR)$. By Proposition \ref{pro1}, $r_R(aR)$ properly maximal this implies $a=b$. Therefore, $a$ is the only vertex adjacent to $1-a$. Thus $1-a$ is pendent vertex.  Conversely, suppose that $a$ is a pendent vertex in $\Gamma^*_s(R)$, so $C(a)$ is non trivial by Corollary \ref{cor2}, $\Gamma^*_s(R)$ is connected. Clearly $a$ and $1-C(a)$ are adjacent in $\Gamma^*(R)$. Since $a$ is a pendent vertex, $1-C(a)$ is a cut vertex.
	\end{proof}

	\begin{corollary}
		Let $R$  be a p.q.-Baer $*$-ring. Then $\Gamma^*_s(R)$ contains a cut vertex if and only if there exists $a\in R$ such that $|r_R(a)|=2$. 
	\end{corollary}
	\begin{proof}
		Suppose that $a$ is a cut vertex in $\Gamma^*_s(R)$. Then as in Theorem \ref{thm5}, $1-a$ is a pendant vertex. Therefore, $|r_R(1-a)|=2$. Conversely, suppose that $|r_R(a)|=2$. Clearly, $a$ is a pendant vertex. By Theorem \ref{thm5}, $\Gamma^*_s(R)$ contains a cut vertex.
	\end{proof}
	The following result is a analogue of results by Redmond {\cite[Corollar 6]{spr2}} and Patil \textit{et al.} {\cite[Theorem 3.9]{Pat2}  }
	\begin{theorem}
		Let $R$ be a p.q.-Baer $*$-ring. Then the set of cut vertices of $\Gamma^*_s(R)$ forms a complete subgraph of $\Gamma^*_s(R)$.
	\end{theorem}
	
	\begin{proof}
		If $\Gamma^*_s(R)$ contains only one cut vertex, then the result is obvious. Suppose that $\Gamma^*_s(R)$ has at least two cut vertices, say $e$ and $f$. Then both $e$ and $f$ are central projections in $R$  and $1-e$, $1-f$ both are pendent vertices in $\Gamma^*_s(R)$. Since, $\Gamma^*_s(R)$ is connected, there is path joining $1-e$ and $1-f$  which definitely contains $e$ and $f$. If $ef\neq 0$, then it is adjacent to both $1-e$ and $1-f$, a contradiction since these are pendent vertices. Hence $e$ and $f$ are adjacent. Hence, the set of cut vertices forms a complete subgraph of  $\Gamma^*_s(R)$. 
	\end{proof}
	
	\section{Complement of Strong Zero-Divisor Graph of a p.q.-Baer $*$-ring}
	In this section, we characterize all p.q.-Baer $*$-rings for which $\Gamma^*_s(R)^c$ is connected.  We use the following known properties of a simple graph.
	\begin{lemma}\label{Lemma of connected graph}
		Let $G$ be a simple graph. Then 
		\begin{enumerate}
			\item $G^c$ is connected if and only if $diam(G)\geq 3$,
			\item $G$ or $G^c$ is connected,
			\item Complement of a complete graph has no edges,
			\item Complement of a complete bipartite graph is disconnected.
		\end{enumerate}
	\end{lemma}
	\begin{example}
		Let $R_1=\mathbb{Z}_6$ and $R_2= \mathbb{Z}_2\times \mathbb{Z}_2\times \mathbb{Z}_2$ be $*$ rings with identity mapping as an involution. Note that $R_1$ and $R_2$ are  p.q.-Baer $*$-rings.
		Observe that, $V(\Gamma^*_s(\mathbb{Z}_6))= \{2, 3, 4\}$ and $V(\Gamma^*_s {(\mathbb{Z}_2\times \mathbb{Z}_2\times \mathbb{Z}_2)})= \{a= (1,0,0), b=(0,1,0), c=(0,0,1), d=(1,1,0), e=(1,0,1), f=(0,1,1)\}$. See fig. 3, the complement of  $\Gamma^*_s(R_1)^c$ is disconnected and complement of  $\Gamma^*_s(R_2)^c$ is  connected.  
	\end{example}

	\begin{center}
		\begin{tikzpicture}[scale=0.7]

			\draw (0,2)--(0,0)--(2,0);
			\draw[fill=black](2,0) circle(.1);
			\draw[fill=black](0,0) circle(.1);
			\draw[fill=black](0,2) circle(.1);
			\node [left ] at(0,2) {$2$};
			\node [left ] at(0,0) {$3$};
			\node [below] at(2,0) {$4$};
			\node [below] at(1,-0.5) {$\Gamma^*_s(\mathbb{Z}_6)$};
			\node [below] at(1,-1.5) {(i)};

			\draw (4,2)--(6,0);
			\draw[fill=black](4,0) circle(.1);
			\draw[fill=black](6,0) circle(.1);
			\draw[fill=black](4,2) circle(.1);
			\node [left ] at(4,2) {$2$};
			\node [left ] at(4,0) {$3$};
			\node [below] at(6,0) {$4$};
			\node [below] at(5,-0.5) {$\Gamma^*_s(\mathbb{Z}_6)^c$};
			\node [below] at(5,-1.5) {(ii)};

			\draw (8,0)--(10,0)--(12,0)--(14,0);
			\draw (11,4)--(11,2)--(12,0);
			\draw (11,2)--(10,0);
			\draw[fill=black](8,0) circle(.1);
			\draw[fill=black](10,0) circle(.1);
			\draw[fill=black](12,0) circle(.1);
			\draw[fill=black](14,0) circle(.1);
			\draw[fill=black](11,4) circle(.1);
			\draw[fill=black](11,2) circle(.1);
			\node [below ] at(8,0) {$e$};
			\node [below ] at(10,0) {$b$};
			\node [below ] at(12,0) {$c$};
			\node [below ] at(14,0) {$d$};
			\node [left ] at(11,2) {$a$};
			\node [left] at(11,4) {$f$};

			\node [below] at (11, -0.5){$ \Gamma^*_s {(\mathbb{Z}_2\times \mathbb{Z}_2\times \mathbb{Z}_2)}$};
			\node [below] at(11,-1.5) {(iii)};
			
			\draw (16,0)--(19,4)--(22,0)--(16,0);
			\draw (17.5,2)--(20.5,2)--(19,0)--(17.5,2);
			\draw[fill=black](16,0) circle(.1);
			\draw[fill=black](17.5,2) circle(.1);
			\draw[fill=black](20.5,2) circle(.1);
			\draw[fill=black](22,0) circle(.1);
			\draw[fill=black](19,4) circle(.1);
			\draw[fill=black](19,0) circle(.1);
			
			\node [below ] at(16,0) {$a$};
			\node [left ] at(17.5,2) {$d$};
			\node [right ] at(20.5,2) {$f$};
			\node [below ] at(22,0) {$c$};
			\node [left ] at(19,4) {$a$};
			\node [below] at(19,0) {$e$};
			
			\node [below] at (19, -0.5){$ \Gamma^*_s ({\mathbb{Z}_2\times \mathbb{Z}_2\times \mathbb{Z}_2})^c$};
			\node [below] at(19,-1.5) {(iv)};
			
			\node [below] at (11, -2.5){$Fig. 3.$};

		\end{tikzpicture}  
	\end{center}

	From the above example, it is clear that the complement of the strong zero-divisor graph of a p.q.-Baer $*$-ring is not connected in general.\\
	In \cite{Visweswaran1}, Visweswaran determined when $\Gamma(R)^c$ is connected for a commutative ring with unity.
	
	\begin{theorem}[{\cite[Theorem 1.1]{Visweswaran1}}]
		Let $R$ be a ring such that $Z(R)^*$ contains at least two elements. Then $(\Gamma(R))^c$ is connected if and only if either (i), (ii) or (iii) holds:\\
		(i) $R$ has exactly one maximal N-prime $P$ of $(0)$ and $P$ is not a $B$-prime of $(0)$.\\
		(ii) $R$ has exactly two maximal N-primes $P_1, P_2$ of $(0)$ and $P_1 \cap P_2 \neq (0)$.\\
		(iii) $R$ has at least three maximal N-primes of $(0)$.
	\end{theorem} 
	In the following result, we prove that the distance between two elements in $\Gamma^*_s(R)$ is the same as the distance between their central covers. We use this result to characterize the connectedness of $\Gamma^*_s(R)^c$.
	\begin{lemma}
		Let $R$ be a p.q.-Baer $*$-ring, and $a, b \in \Gamma^*_s(R)$ Then  $d(a,b)=n$ if and only if $d(C(a), C(b))=n$.
	\end{lemma}
	\begin{proof}
		Suppose $d(a,b)=n$ in $V(\Gamma^*_s(R))$. Let $a=a_1\sim a_2 \sim \cdots \sim  a_{n+1}=b$ be a path of length $n$ joining $a$ and $b$. Therefore, $C(x)\notin \{0, 1\}$ for $x\in \{a_1, a_2, \cdots , a_{n+1}\}$ and $C(a_i)C(a_j)=0$ if and only if $j=i+1$, for $i \in \{1,2,\cdots,  n\}$. Hence $C(a_1)\sim  C(a_2) \sim  \cdots \sim  C(a_{n+1})$ is a path. Observe that $C(a_i)\neq C(a_j)$, for $i\neq j$, otherwise we get a path of length less than $n$ which connects $a$ and $b$. Since $a$ and $b$ are non-adjacent, we have $C(a)C(b)\neq 0$, Hence $C(a)$ and $C(b)$ are non-adjacent.  Suppose there is a path of length less than $n$ which connects $C(a)$ and $C(b)$.  Since $a$  and $C(a)$ are adjacent to the same vertices, we get a path of length less than $n$, which connects $a$ and  $b$, a contradiction. Hence, $d(C(a), C(b))=n$. Conversely, suppose that $d(C(a), C(b))=n$. Let $C(a)\sim b_1\sim b_2\sim \cdots \sim b_{n-1} \sim C(b)$, Clearly  $a\sim b_1\sim b_2 \sim \cdots  \sim b_{n-1} \sim b$  is a path, and as proved above, there is no path of length less than $n$  which connects $a$ and $b$. Therefore $d(a, b)=n$.
	\end{proof}
	\noindent For any central projection $e$ in a p.q.-Baer $*$-ring $R$, we denote $C_e=\{a\in R~|~ C(a)=e\}$.
	
	\begin{proposition}{\label{pro2}}
		Let $R$ be a p.q.-Baer $*$-ring.  If $|CP(R)|=4$, then $\Gamma^*_s(R)$ is complete bipartite, hence $\Gamma^*_s(R)^c$ is disconnected.
	\end{proposition}
	
	\begin{proof}
		Suppose $CP(R)=\{0,1,e, 1-e\}$. Then $C_e, C_{1-e}\subseteq V(\Gamma^*_s(R))$. Therefore in $\Gamma^*_s(R)$, each element in $C_e$ is adjacent to every element to $C_{1-e}$. Hence $\Gamma^*_s(R)$ is complete bipartite. By Lemma \ref{Lemma of connected graph}, $\Gamma^*_s(R)^c$ disconnected.
	\end{proof}
	
	\begin{lemma}{\label{lm2}}
		In a p.q.-Baer $*$-ring for nontrivial central projections $e, f$ with $e+f\neq 1$, if $ef=0$, then $d(1-e, 1-f)=3$.
	\end{lemma}
	
	\begin{proof}
		Let $e, f$ be two non-trivial central projection in a p.q.-Baer $*$-ring $R$ with $e+f \neq 1$.
		If $ef=0$ then $e$ is adjacent to $f$. We know that $e$ is adjacent to $1-e$ and $f$ is adjacent to $1-f$. Hence $1-e \sim e \sim f \sim 1-f $ is a path of length $3$. Clearly, $(1-e)(1-f)\neq 0$, if possible $(1-e)(1-f)= 0$, implies that $1-f-e-ef=0$, i.e. $e+f=1$ which is a contradiction to $e+f \neq 1$. 
		If $1-e \sim a \sim 1-f$, then  $(1-e)\sim C(a) \sim 1-f$. Then $c(a)(1-e)=0$ and $c(a)(1-f)=0$. Which gives $c(a)=c(a)e=c(a)f$. This implies  $c(a)=c(a)f=c(a)ef=0$. So $a=c(a)a=0$, which is contradiction to $a \in V(\Gamma^*_s(R))$. Hence, there is no path of length $2$ joining $1-e$ and $1-f$. Thus, $d(1-e, 1-f)=3$.
	\end{proof}
	
	In the following result, we characterize all p.q.-Baer $*$-rings for which $\Gamma^*_s(R)^c$ is connected. 
	
	\begin{theorem} \label{th6}
		For a p.q.-Baer $*$-ring $R$, $\Gamma^*_s(R)^c$ is connected if and only if $|CP(R)|\geq 6$.
	\end{theorem}
	\begin{proof}
		Suppose that $\Gamma^*_s(R)^c$ is connected. Observe that the sets $\{e, 1-e\}_{e\in CP(R)}$ forms a partition of $CP(R)$. Therefore, if  $|CP(R)|$ is infinite, then $|CP(R)|\geq 6$. If $|CP(R)|$ is finite, then it is even. If $|CP(R)|=2$, then $V(\Gamma^*_s(R))=\emptyset$.
		If $|CP(R)|=4$,  by Proposition \ref{pro2}, $\Gamma^*_s(R)^c$ is disconnected, which is a contradiction. Hence, $|CP(R)|\geq 6$. 	
		
		Conversely, suppose that $|CP(R)|\geq 6$. If graph $\Gamma^*_s(R)$ is disconnected, then by Lemma \ref{Lemma of connected graph}, $\Gamma^*_s(R)^c$ is connected. Suppose $\Gamma^*_s(R)$ is connected. Now, to show $\Gamma^*_s(R)^c$ is connected, for this, it is sufficient to show $diam(\Gamma^*_s(R))=3$. Let $e$ be any atom in the lattice of central projections in $R$. In view of Lemma \ref{lm2}, it is enough to show that there exists a nontrivial central projection $f$ such that $ef=0$ with $e+f\neq 1$. On the contrary, assume that $ef=0$ and $e+f=1$. Let $h$ be any projection in $R$ such that $h\notin \{0, 1, e, 1-e\}$ (such $h$ exists because $|CP (A)| \geq 6$). Since $h \neq 1-e$, we have $eh\neq 0$. If $(1-e)h = 0$, then we get $h = he = eh$, i.e. $h \leq e$. Since $e$ is an atom, we must have either $h = 0$ or $h = e$, a contradiction to the choice of $h$. Hence $(1-e)h \neq 0$. Since $\Gamma^*_s(R)$ is connected, there is a path joining $e$ and $h$. Suppose $e \sim a_1 \sim a_2 \sim \cdots \sim a_{k-1}\sim a_k = h$ is a path of length $k$ joining $e$ and $h$. Then $eC(a_1) = 0$. By assumption, we have $C(a_1 ) = 1-e$. Also, $C(a_1 )C(a_2) = 0$ which yields $(1-e)C(a_2 ) = 0$. This gives $C(a_2 ) \leq e$. Since, $e$ is an atom and $C(a_2)=0$, we get $C(a_2)=e$. Continuing in this way, we get $C(a_i ) = e$, if $i$ is even and $C(a_i ) = 1-e$, if $i$ is odd. Thus, $h = e$ or $h = 1-e$, a contradiction in both cases. Hence, there exists a central nontrivial projection $f$ such that $ef = 0$ with $e + f \neq 1$. By Lemma \ref{lm2},  $d(1-e, 1-f)=3$. Hence, by Lemma \ref{Lemma of connected graph}, $\Gamma^*(R)^c$ connected.   
	\end{proof}
	
	\begin{corollary}
		Let $R$ be a p.q.-Baer $*$-ring. If $\Gamma^*_s(R)$ contains a vertex adjacent to all the other vertices then $|CP(R)| = 4$.
	\end{corollary}
	
	\begin{proof}
		If $a$ is a vertex adjacent to
		all the other vertices in  $\Gamma^*_s(R)$ , then $a$ is  isolated vertex in $\Gamma^*_s(R)^c$. Hence  $\Gamma^*_s(R)^c$ is disconnected, by Theorem \ref{th6},	$|CP(R)|=4$.
	\end{proof}
	
	\begin{proposition}
		Let $R$ be a p.q.-Baer $*$-ring such that $\Gamma^*_s(R)^c$ is disconnected, then $gr(\Gamma^*_s(R)^c )\in \{3, \infty\}$.
		
	\end{proposition}
	
	\begin{proof}
		Since $\Gamma^*_s(R)^c $ is disconnected, by Theorem \ref{th6} , $|CP(R)|=4$. Let $CP(R) =\{0, 1, e, 1-e\}$, $C_e =\{x \in R ~|~ C(x) = e\}$ and $C_{1-e}=\{x \in R~ | ~ C(x) = 1-e\}$. Clearly, $E_e$ and $E_{1-e}$ are two connected components of $\Gamma^*_s(R)^c$. If $|C_e|\geq 3$ or $|C_{1-e} | \geq 3$ then $gr(\Gamma^*_s(R))=3$. Otherwise $gr(\Gamma^*_s(R)^c ) =\infty$.
	\end{proof}
	
	\begin{lemma}{\label{lm4}}
		Let $R$ be a p.q.-Baer $*$-ring such that, $\Gamma^*_s(R)$ is triangle
		free. then $a$ and $b$ are adjacent in $\Gamma^*_s(R)$ if and only if $C(a)=1-C(b)$.
	\end{lemma}
	
	\begin{proof}
		Suppose $a \sim b$, then $C(a)C(b)=0$. Hence $C(a)+C(b)$ is a central projection in $R$. If $C(a)\neq(1-C(b))$, then $e=1-(C(a)+C(b))$ is a nontrivial central projection and $C(a)\sim C(b)\sim e \sim C(a)$ is a 3-cycle, a contradiction. Hence, $C(a)=1-C(b)$.
	\end{proof}
	
	\begin{theorem}[{\cite[Theorem 3.2]{ NMK1}}]{\label{thm10}}
		Let $R$ be a p.q.-Baer $*$-ring. Then 
		\begin{enumerate}
			\item [(A)] the following statements are equivalent.
			\begin{enumerate}
				\item [(i)] $gr(\Gamma^*_s(R))=3$.
				\item[(ii)] $R$ contains two nontrivial central projections $e$ and $f$ such that $ef=0$ and $e+f\neq 1$.
				\item[(iii)] $L(CP(R))$ contains a 3-element chain not containing 0.
			\end{enumerate}
			\item[(B)]  $gr(\Gamma^*_s(R))=4$ if and only if $\Gamma^*_s(R)$ is triangle-free and $R$ contains at least one nontrivial central projection $e$ such that $|C_e|\geq 2$ and $|C_{1-e}|\geq 2$.
		\end{enumerate}
	\end{theorem}
	
	In the following result, we determine the diameter and girth of $\Gamma^*_s(R)^c$, whenever it is connected.
	
	\begin{theorem}\label{thm11}
		Let $R$ be a p.q.-Baer $*$-ring. If $\Gamma^*_s(R)^c$ is connected, then $diam(\Gamma^*_s(R)^c)=2$ and $gr(\Gamma^*_s(R)^c)=3$.
	\end{theorem}
	
	\begin{proof}
		If $\Gamma^*_s(R)^c$ is connected, by Theorem \ref{th6}, $|CP(R)\geq 6|$. Let $a, b \in V(\Gamma^*_s(R)^c)$. If $C(a)C(b)\neq 0$, then they are adjacent in $\Gamma^*_s(R)^c$. 
		
		If $C(a)C(b)=0$, then $g=C(a)+C(b)$ is a central projection in $R$. If $g<1$, then $aRg=aR\neq 0$ and $bRg=bR\neq 0$, therefore $g$ is adjacent to both $a$ and $b$ in $\Gamma^*_s(R)^c$, which gives $a\sim g \sim b$ is path of length $2$  in $\Gamma^*_s(R)^c$. Suppose $g=1$, let $f$ be a central projection in $R$ such that $f\notin \{0,1, C(a), C(b)\}$. If $C(a)f=0$, then $(1-f)C(a)\neq 0$ and $(1-f)C(b)\neq 0$ (by selection of $f$). Thus $a\sim (1-f) \sim b$	is a path of length $2$ in $\Gamma^*_s(R)^c.$ Similarly we get a path of length $2$ if $C(b)f=0$. If $C(a)f\neq 0$ and $C(b)f\neq 0$ then $a \sim f \sim b$ is a path of length $2$ in $\Gamma^*_s(R)^c$. Thus in any case $d(a, b)=2$, hence $diam(\Gamma^*_s(R)^c)=2$.
		
		Now we  prove $gr(\Gamma^*_s(R)^c)=3$.
		
		\noindent\textbf{Case i:} If $gr(\Gamma^*_s(R))=3$. by Theorem \ref{thm10}, $R$ contains two nontrivial central projections $e, f$ such that $ef=0$ and $e+f\neq 1$, which gives  $e(e+f)=e+ef=e\neq 0$, $(e+f)(1-f)=e-ef+f-f=1\neq 0$ and $(1-f)e\neq 0$. Thus $e\sim (e+f) \sim (1-f) \sim e$ is a $3$-cycle in $\Gamma^*_s(R)^c$, hence $gr(\Gamma^*_s(R)^c)=3$.
		
		\noindent\textbf{case ii:} If $gr(\Gamma^*_s(R))\neq 3$ i.e. $\Gamma^*_s(R)$ is triangle free. By Lemma \ref{lm4}, for nontrivial central projections $e, f\in R$, we have $ef=0$ if and only if $e+f=1$. Let $g, h\in CP(R)$ such that $g+h\neq 1$, then $gh\neq 0$. If $ gR(gh)^*=gRhg=ghR \neq 0 $, then $g\sim gh \sim h \sim g$ is a $3$-cycle in  $\Gamma^*_s(R)^c$. If $ gR(gh)^*=gRhg=ghR = \{0\}$, then $hg =0$ which is a contradiction to $hg \neq 0$. Hence, in both cases $gr(\Gamma^*_s(R)^c)=3$.
	\end{proof}
	For distinct vertices $a$ and $b$ of $G$, $a$ and $b$
	are said to be {\it{orthogonal}}, if $a$ and
	$b$ are adjacent, and they don't have any common neighbor; it is denoted by $a\perp b$. We say
	that  $G$ is {\it{complemented}}  if for each vertex $a$ of $G$,
	there is a vertex $b$ of $G$ (called a {\it{complement}} of $a$)
	such that $a\perp b$.
	\begin{lemma}\label{lm5}
		Let $R$ be a p.q.-Baer $*$-ring such that $\Gamma^*_s(R)^c$ is connected, then it is not complemented.
	\end{lemma}
	\begin{proof}
		Suppose that $\Gamma^*_s(R)^c$ is connected, then $R$ contains at least $6$ central projections. Let $a,b$ are adjacent in $\Gamma^*_s(R)^c$.
		If $C(a)\neq C(b)$ then $C(a)C(b)\neq 0$ this gives $C(a)$ and $C(b)$ adjacent in $\Gamma^*_s(R)^c$.
		also $aC(a)C(b)=aC(b)\neq 0$ and $bC(a)C(b)=bC(a)\neq 0$ i.e. $C(a)C(b)$  is adjacent to both $a$ and $b$ in $\Gamma^*_s(R)^c$.
		If $C(a)=C(b)=e$, then choose central projection $f\notin \{0, 1, e, 1-e\}$.
		If $ef\neq 0$ then $e$ and $f$ are adjacent in $\Gamma^*_s(R)^c$ \textit{i.e.} $C(a)f\neq 0$ and $C(b)f\neq 0$ which gives $a\sim f \sim b$. Hence $f$ is adjacent to both $a$ and $b$ in $\Gamma^*_s(R)^c$. If $ef=0$, then $e$ and $f$ are non adjacent in $\Gamma^*_s(R)^c$, \textit{i.e.} $e(1-f)=e\neq 0$ implies $a(1-f)\neq 0$ and  $b(1-f)\neq 0$. Therefore $a \sim (1-f) \sim b$ in $\Gamma^*_s(R)^c$. Hence, for any adjacent vertex $a$ and $ b$ in $\Gamma^*_s(R)^c$, there is a vertex that is adjacent to both $a$ and $b$. Thus, $\Gamma^*_s(R)^c$ is not complemented.
	\end{proof}
	
	\begin{theorem}
		Let $R$ be a p.q.-Baer $*$-ring and there is an element $a\in V(\Gamma^*_s(R))$ be such that $a\neq C(a)$. Then $\Gamma^*_s(R)^c$ is complemented if and only if $R=\mathbb{Z}_3\oplus \mathbb{Z}_3.$
	\end{theorem}
	
	\begin{proof}
		If $R=\mathbb{Z}_3\oplus \mathbb{Z}_3$, then $\Gamma^*_s(R)$ is a $4$-cycle whose graph complement is complemented. Conversely, suppose $\Gamma^*_s(R)^c$ is complemented. Since $a\in V(\Gamma^*_s(R))$ be such that $a\neq C(a)$, therefore $a$ is not adjacent ot $C(a)$. Hence $\Gamma^*_s(R)$ is not complete. By Lemma \ref{lm5}, $\Gamma^*_s(R)^c$ is disconnected. Therefore, $R$ contains exactly four central projections. Let $CP(R)=\{0,1,e, 1-e\}$,  $C_e =\{x \in R ~|~ C(x) = e\}$ and $C_{1-e}=\{x \in R ~|~ C(x) = 1-e\}$. If $|C_e|>2$,  then any three elements in $C_e$ are connected in $\Gamma^*_s(R)^c$, hence no element of $C_e$ has complement, a contradiction. Therefore $|C_e|\leq 2$. Similarly $|C_{1-e}|\leq 2$. If $|C_e|=|C_{1-e}|=1$, then $\Gamma^*_s(R)$ is a complete graph on two vertices, a contradiction to the assumption. Therefore  $|C_e|=|C_{1-e}|=2$. As $e$ is a central projection, $R=eR\oplus (1-e)R$. Where $|eR|=|(1-e)R|=3$ (as $C_e=eR\backslash \{0\}$ and $C_{1-e}=(1-e)R\backslash \{0\} $). Therefore $R=\mathbb{Z}_3\oplus \mathbb{Z}_3$.
	\end{proof}

\end{document}